\documentclass{amsart}
\usepackage{latexsym}
\usepackage{amstext}
\usepackage{amsmath}
\usepackage{amssymb}
\usepackage{amsthm}
\usepackage{url}
\usepackage{enumerate}

\newtheorem{theorem}{Theorem}[section]
\newtheorem{lemma}[theorem]{Lemma}

\theoremstyle{definition}

\newtheorem{question}[theorem]{Question}

\def \kbar {\overline{k}}
\def \Qbar {\overline{\mathbb{Q}}}
\def \O {\mathcal{O}}
\DeclareMathOperator{\Sym}{Sym}
\DeclareMathOperator{\Supp}{Supp}
\DeclareMathOperator{\im}{im}
\DeclareMathOperator{\End}{End}
\DeclareMathOperator{\codim}{codim}

\begin{document}
\bibliographystyle{amsplain}
\title{Wirsing-type inequalities}
\author{Aaron Levin}
\address{Department of Mathematics\\Michigan State University\\East Lansing, MI 48824}
\curraddr{}
\email{adlevin@math.msu.edu}

\begin{abstract}
Wirsing's theorem on approximating algebraic numbers by algebraic numbers of bounded degree is a generalization of Roth's theorem in Diophantine approximation.  We study variations of Wirsing's theorem where the inequality in the theorem is strengthened, but one excludes a certain easily-described special set of approximating algebraic points.
\end{abstract}

\maketitle

\section{Introduction}

Roth's fundamental result in Diophantine approximation describes how closely an algebraic number may be approximated by rational numbers:
\begin{theorem}[Roth \cite{Roth}]
\label{TRoth}
Let $\alpha\in \Qbar$ be an algebraic number.  Let $\epsilon>0$.  Then there are only finitely many rational numbers $\frac{p}{q}\in \mathbb{Q}$ satisfying
\begin{equation*}
\left|\alpha-\frac{p}{q}\right|<\frac{1}{q^{2+\epsilon}}.
\end{equation*}
\end{theorem}

Roth's theorem can be extended \cite{Lan, Rid} to an arbitrary fixed number field $k$ (in place of $\mathbb{Q}$) and to allow finite sets of absolute values (including non-archimedean ones).  A general statement of Roth's theorem, using the language of heights (see Section \ref{sDio} for the definitions), is the following.

\begin{theorem}
\label{tRoth}
Let $S$ be a finite set of places of a number field $k$.  Let $P_1,\ldots, P_q\in \mathbb{P}^1(k)$ be distinct points, $D=\sum_{i=1}^q P_i$, and $\epsilon>0$.  Then for all but finitely many points $P\in \mathbb{P}^1(k)\setminus \Supp D$,
\begin{equation*}
m_{D,S}(P)=\sum_{i=1}^q\sum_{v\in S}h_{P_i,v}(P)< (2+\epsilon)h(P).
\end{equation*}
\end{theorem}

We note that there is no loss of generality in the assumption that $P_1,\ldots, P_q$ are $k$-rational (see \cite[Remark 2.2.3]{V2}).

Instead of taking the approximating elements from a fixed number field, a natural variation on Roth's theorem is to consider approximation by algebraic numbers of bounded degree.  In this direction, Wirsing \cite{Wir} proved a generalization of Roth's theorem, which we state in a general form.

\begin{theorem}[Wirsing]
\label{tWir}
Let $S$ be a finite set of places of a number field $k$.  Let $P_1,\ldots, P_q\in \mathbb{P}^1(k)$ be distinct points and let $D=\sum_{i=1}^q P_i$.  Let $\epsilon>0$ and let $d$ be a positive integer.  Then for all but finitely many points $P\in \mathbb{P}^1(\kbar)\setminus\Supp D$ satisfying $[k(P):k]\leq d$,
\begin{equation*}
m_{D,S}(P)< (2d+\epsilon)h(P).
\end{equation*}
\end{theorem}

Taking $d=1$ in Wirsing's theorem recovers Roth's theorem.  For $t\leq 2d$ and $D$, $S$, $k$, as in Theorem \ref{tWir}, the set
\begin{equation}
\label{iset}
\{P\in\mathbb{P}^1(\kbar)\mid [k(P):k]=d, m_{D,S}(P)\geq th(P)\}
\end{equation}
may be infinite.  A natural way to obtain algebraic points $P\in\mathbb{P}^1(\kbar)$ with $[k(P):k]=d$ is to pull back $k$-rational points via a degree $d$ morphism $\phi:\mathbb{P}^1\to\mathbb{P}^1$.  The following result may be used to classify those morphisms $\phi$ which contribute infinitely many points in this way to the set \eqref{iset}.

\begin{theorem}
\label{tbor}
Let $S$ be a finite set of places of a number field $k$ containing the archimedean places.  Let $P_1,\ldots, P_q\in \mathbb{P}^1(k)$ be distinct points and let $D=\sum_{i=1}^q P_i$.  Let $\phi:\mathbb{P}^1\to\mathbb{P}^1$ be a morphism over $k$ of degree $d$.  Let $\phi(\{P_1,\ldots, P_q\})=\{Q_1,\ldots, Q_r\}$ and let
\begin{align*}
n_i=|\phi^{-1}(Q_i)\cap\{P_1,\ldots, P_q\}|, \quad i=1,\ldots, r.
\end{align*}
Rearrange the indices so that $n_1\geq n_2\geq \cdots \geq n_r$.
\begin{enumerate}
\item 
\label{p1}
Suppose that $|S|>1$.  For some constant $C$, the inequality
\begin{align*}
m_{D,S}(P)>(n_1+n_2)h(P)-C
\end{align*}
holds for infinitely many points $P\in \phi^{-1}(\mathbb{P}^1(k))$.
\item 
\label{p2}
Let $\epsilon>0$.  The inequality
\begin{align*}
m_{D,S}(P)<(n_1+n_2+\epsilon)h(P)
\end{align*}
holds for all but finitely many points $P\in \phi^{-1}(\mathbb{P}^1(k))$ with $[k(P):k]=d$.
\end{enumerate}
\end{theorem}

After composing $\phi$ with an automorphism, we can always assume in Theorem~\ref{tbor} that $Q_1=0$ and $Q_2=\infty$.  Then Theorem \ref{tbor} motivates making the following definitions.  Let $k$ be a number field, $P_1,\ldots, P_q\in \mathbb{P}^1(k)$ be distinct points, and $D=\sum_{i=1}^qP_i$.  Let $d$ be a positive integer and let $t$ be a positive real number.  Let $\End_k(\mathbb{P}^1)$ be the set of $k$-morphisms $\phi:\mathbb{P}^1\to\mathbb{P}^1$.  Define
\begin{align*}
\Phi(D,d,t,k)&=\{\phi\in \End_k(\mathbb{P}^1)\mid \deg \phi\leq d, |\phi^{-1}(\{0,\infty\})\cap \Supp D|\geq t\},\\
Z(D,d,t,k)&=\bigcup_{\phi\in \Phi(D,d,t,k)} \phi^{-1}(\mathbb{P}^1(k)).
\end{align*}
It is then natural to ask the following question.

\begin{question}
\label{mq}
Does the inequality
\begin{equation}
\label{Wineq}
m_{D,S}(P)< th(P)
\end{equation}
hold for all but finitely many points $P\in \mathbb{P}^1(\kbar)\setminus Z(D,d,t,k)$ satisfying $[k(P):k]\leq d$?
\end{question}

We will show that Question \ref{mq} has a positive answer when $d=2$.
\begin{theorem}
\label{tQ}
Let $S$ be a finite set of places of a number field $k$.  Let $P_1,\ldots, P_q\in \mathbb{P}^1(k)$ be distinct points, let $D=\sum_{i=1}^qP_i$, and let $t$ be a positive real number.  Then the inequality
\begin{equation*}
m_{D,S}(P)< th(P)
\end{equation*}
holds for all but finitely many points $P\in \mathbb{P}^1(\kbar)\setminus Z(D,2,t,k)$ satisfying $[k(P):k]\leq 2$.
\end{theorem}

More generally, we will show that Question \ref{mq} has a positive answer if either $t\leq d+1$ (Lemma \ref{rem}) or $t>2d-1$: 
\begin{theorem}
\label{Wthm}
Let $S$ be a finite set of places of a number field $k$.  Let $P_1,\ldots, P_q\in \mathbb{P}^1(k)$ be distinct points and let $D=\sum_{i=1}^qP_i$.  Let $d$ be a positive integer and let $t>2d-1$ be a real number.  Then the inequality
\begin{equation*}
m_{D,S}(P)< th(P)
\end{equation*}
holds for all but finitely many points $P\in \mathbb{P}^1(\kbar)\setminus Z(D,d,t,k)$ satisfying $[k(P):k]\leq d$.  Furthermore, in this case $\Phi(D,d,t,k)$ is a finite set and $Z(D,d,t,k)=\bigcup_{\phi\in \Phi(D,d,t,k)} \phi^{-1}(\mathbb{P}^1(k))$ is a finite union of sets of the form $\phi^{-1}(\mathbb{P}^1(k))$.
\end{theorem}

Thus, after excluding points of a special and easily described form, the inequality in Wirsing's theorem may be improved to
\begin{equation*}
m_{D,S}(P)< (2d-1+\epsilon)h(P).
\end{equation*}

In general, we will see that Question \ref{mq} has a negative answer.  By carefully studying the exceptional hyperplanes in the Schmidt Subspace Theorem in dimension three, we obtain a precise answer to Question \ref{mq} when $d=3$, showing that in this case the question has a positive answer if $t>\frac{9}{2}$, but (at least for some choices of the parameters) it has a negative answer when $4<t<\frac{9}{2}$.

\begin{theorem}
\label{td3}
Let $k$ be a number field.  Let $P_1,\ldots, P_q\in \mathbb{P}^1(k)$ be distinct points and let $D=\sum_{i=1}^qP_i$.  Let $S$ be a finite set of places of $k$ containing the archimedean places and let $t$ be a real number.
\begin{enumerate}
\item  \label{Wa} If $t>\frac{9}{2}$, then the inequality
\begin{equation*}
m_{D,S}(P)< th(P)
\end{equation*}
holds for all but finitely many points $P\in \mathbb{P}^1(\kbar)\setminus Z(D,3,t,k)$ satisfying $[k(P):k]\leq 3$.
\item \label{Wb} If $4<t<\frac{9}{2}$, $|S|>2$, and $q=6$, then there are infinitely many points $P\in \mathbb{P}^1(\kbar)\setminus Z(D,3,t,k)$ satisfying $[k(P):k]=3$ and
\begin{equation*}
m_{D,S}(P)> th(P).
\end{equation*}
\end{enumerate}
\end{theorem}

From another viewpoint, Question \ref{mq} may be viewed as asking a quantitative generalization of results in \cite{Lev5}, where integral points of bounded degree on affine curves were studied.  In \cite{Lev5}, affine curves with infinitely many integral points of degree $d$ (over some number field) were characterized as follows.

\begin{theorem}
\label{tm}
Let $C\subset \mathbb{A}^n$ be a nonsingular affine curve defined over a number field $k$.  Let $\tilde{C}$ be a nonsingular projective completion of $C$ and let $(\tilde{C}\setminus C)(\kbar)=\{P_1,\ldots, P_q\}$.  Let $d$ be a positive integer.  Let $\overline{\O}_{k,S}$ denote the integral closure of $\O_{k,S}$ in $\kbar$.  Then there exists a finite extension $L$ of $k$ and a finite set of places $S$ of $L$ such that the set
\begin{align*}
\{P\in C(\overline{\O}_{L,S})\mid [L(P):L]\leq d\}
\end{align*}
is infinite if and only if there exists a morphism $\phi:\tilde{C}\to\mathbb{P}^1$, over $\kbar$, with $\deg \phi\leq d$ and $\phi(\{P_1,\ldots, P_q\})\subset \{0,\infty\}$.
\end{theorem}

When $C=\mathbb{P}^1$ the following stronger result was proven.
\begin{theorem}
\label{tP}
Let $S$ be a finite set of places of a number field $k$ containing the archimedean places.  Let $P_1,\ldots, P_q\in \mathbb{P}^1(k)$ be distinct points and let $D=\sum_{i=1}^q P_i$.   Let $d$ be a positive integer.  For any set of $(D,S)$-integral points $R\subset\{P\in C(\kbar)\mid [k(P):k]\leq d\}$, the set $R\setminus Z(D,d,q,k)$ is finite.
\end{theorem}
Note that $\Phi(D,d,q,k)$ is just the set of $k$-endomorphisms $\phi$ of $\mathbb{P}^1$ satisfying $\deg \phi\leq d$ and $\phi(\{P_1,\ldots, P_q\})\subset \{0,\infty\}$.  From the definition, $R$ is a set of $(D,S)$-integral points if and only if
\begin{align*}
m_{D,S}(P)=(\deg D)h(P)+O(1)
\end{align*}
for all $P\in R$.  For some finite set of places $T\supset S$, we even have (using the definition of $m_{D,T}$ in Section \ref{sDio}) $m_{D,T}(P)=(\deg D)h(P)$ for all $P\in R$.  Thus, Theorem \ref{tP} is equivalent to Question \ref{mq} having a positive answer for $t=\deg D$.  In this sense, Question \ref{mq} asks a quantitative generalization of Theorem \ref{tm} (for the projective line) and Theorem \ref{tP}.

Similar to Question \ref{mq}, the analogue of Theorem \ref{tm} for algebraic points of bounded degree on curves holds only for small $d$ ($d\leq 3$) as we now discuss.  Let $C$ be a nonsingular projective curve defined over a number field $k$.  Faltings' theorem asserts that $C(L)$ is infinite for some finite extension $L$ of $k$ if and only if the genus of $C$ is zero or one.  If $C$ admits a degree $d$ morphism to the projective line or an elliptic curve, then by pulling back $k$-rational points via this morphism one sees that, after possibly replacing $k$ by a larger number field, the set
\begin{align*}
\{P\in C(\kbar)\mid [k(P):k]\leq d\}
\end{align*}
is infinite.  Harris and Silverman \cite{HS} proved the converse in the case $d=2$.

\begin{theorem}[Harris, Silverman]
Let $C$ be a nonsingular projective curve defined over a number field $k$.  Then the set 
\begin{align*}
\{P\in C(\kbar)\mid [L(P):L]\leq 2\}
\end{align*}
is infinite for some finite extension $L$ of $k$ if and only if $C$ is hyperelliptic or bielliptic.
\end{theorem}

More generally, we have the following theorem of Abramovich and Harris \cite{AH}.

\begin{theorem}[Abramovich, Harris]
\label{AHthm}
Let $d\leq 4$ be a positive integer.  Let $C$ be a nonsingular projective curve over a number field $k$ with genus not equal to $7$ if $d=4$.  Then the set
\begin{align*}
\{P\in C(\kbar)\mid [L(P):L]\leq d\}
\end{align*}
is infinite for some finite extension $L$ of $k$ if and only if $C$ admits a map of degree $\leq d$, over $\kbar$, to $\mathbb{P}^1$ or an elliptic curve.
\end{theorem}

Given Theorem \ref{AHthm}, Abramovich and Harris naturally conjectured that a similar result would hold for all $d$ (this is the analogue of Theorem \ref{tm} for algebraic points).  However, Debarre and Fahlaoui \cite{DF} gave counterexamples to the conjecture for all $d\geq 4$.  The failure of this conjecture and the failure of Question \ref{mq} to always have a positive answer are somewhat analogous.  Debarre and Fahlaoui's counterexamples rely on the fact that there may exist an elliptic curve  $E$ in the Jacobian of a curve $C$ that is not induced by any morphism $C\to E$.  To every morphism $\phi\in\Phi(D,d,t,k)$ of degree $d$, one may associate a line in $\Sym^d\mathbb{P}^1$ via the one-dimensional linear system associated to $\phi$.  Our examples rely on the fact that in a Diophantine approximation problem on $\Sym^d\mathbb{P}^1\cong \mathbb{P}^d$ related to Question \ref{mq}, there are exceptional hyperplanes in the Subspace Theorem that are not induced by the morphisms in $\Phi(D,d,t,k)$, i.e., that are not a Zariski closure of a union of lines associated to morphisms in $\Phi(D,d,t,k)$.

\section{Diophantine approximation on projective space: definitions and background material}
\label{sDio}

Let $k$ be a number field and let $\O_k$ denote the ring of integers of $k$.  Recall that we have a canonical set $M_k$ of places (or absolute values) of $k$ consisting of one place for each prime ideal $\mathfrak{p}$ of $\mathcal{O}_k$, one place for each real embedding $\sigma:k \to \mathbb{R}$, and one place for each pair of conjugate embeddings $\sigma,\overline{\sigma}:k \to \mathbb{C}$.  If $S$ is a finite set of places of $k$ containing the archimedean places,  we let $\O_{k,S}$, and $\O_{k,S}^*$ denote the ring of $S$-integers of $k$ and the group of $S$-units of $k$, respectively.   If $v$ is a place of $k$ and $w$ is a place of a field extension $L$ of $k$, then we say that $w$ lies above $v$, or $w|v$, if $w$ and $v$ define the same topology on $k$.  We normalize our absolute values so that $|p|_v=\frac{1}{p}$ if $v$ corresponds to $\mathfrak{p}$ and $\mathfrak{p}$ lies above a rational prime $p$, and $|x|_v=|\sigma(x)|$ if $v$ corresponds to an embedding $\sigma$.  For $v\in M_k$, let $k_v$ denote the completion of $k$ with respect to $v$.  We set
\begin{equation*}
\|x\|_v=|x|_v^{[k_v:\mathbb{Q}_v]/[k:\mathbb{Q}]}.
\end{equation*}
A fundamental equation is the product formula
\begin{equation*}
\prod_{v\in M_k}\|x\|_v=1,
\end{equation*}
which holds for all $x\in k^*$.

For a point $P=(x_0,\ldots,x_n)\in \mathbb{P}^n(k)$, we have the absolute logarithmic height
\begin{equation*}
h(P)=\sum_{v\in M_k} \log \max\{\|x_0\|_v,\ldots,\|x_n\|_v\}.
\end{equation*}
Note that this is independent of the number field $k$ and the choice of coordinates $x_0,\ldots, x_n\in k$.  In general, one can define a height $h_D$ (and local height $h_{D,v}$, $v\in M_k$), unique up to a bounded function, with respect to any Cartier divisor $D$ on a projective variety (in fact, this can even be done with respect to an arbitrary closed subscheme \cite{Sil}).  If $D$ and $E$ are Cartier divisors on a projective variety $X$, then heights satisfy the additive relation
\begin{align*}
h_{D+E}(P)=h_{D}(P)+h_E(P)+O(1).
\end{align*}
Let $\Supp D$ denote the support of the divisor $D$.  If $\phi:Y\to X$ is a morphism of projective varieties with $\phi(Y)\not\subset \Supp D$, then
\begin{align*}
h_{D}(\phi(P))=h_{\phi^*D}(P)+O(1).
\end{align*}
Similar relations hold for local heights.  We refer the reader to \cite{BG, HinS, Lan, V2} for further details and properties of heights.

We will primarily use heights with respect to effective divisors on projective space.  These can be explicitly described as follows.  Let $D$ be a hypersurface in $\mathbb{P}^n$ defined by a homogeneous polynomial $f\in k[x_0,\ldots, x_n]$ of degree $d$.  For $v\in M_k$, we let $|f|_v$ denote the maximum of the absolute values of the coefficients of $f$ with respect to $v$.  We define $\|f\|_v$ similarly.  For $v\in M_k$ and $P=(x_0,\ldots, x_n)\in\mathbb{P}^n(k)\setminus \Supp D$, $x_0,\ldots, x_n\in k$, we define the local height function
\begin{equation*}
h_{D,v}(P)=\log \frac{\|f\|_v\max_i \|x_i\|_v^d}{\|f(P)\|_v}.
\end{equation*}
Note that this definition is independent of the choice of the defining polynomial $f$ and the choice of the coordinates for $P$. Let $h_D(P)=\sum_{v\in M_k}h_{D,v}(P)$.  It follows from the product formula that $h_D(P)=(\deg D)h(P)$.  Let $S$ be a finite set of places of $k$.  For $P\in \mathbb{P}^n(\kbar)\setminus \Supp D$ we define the proximity function $m_{D,S}(P)$ by
\begin{equation*}
m_{D,S}(P)=\sum_{v\in S}\sum_{\substack{w\in M_{k(P)}\\w|v}}h_{D,w}(P).
\end{equation*}

We will also have occasion to use heights associated to points in projective space.  If $P=(x_0,\ldots, x_n), Q=(y_0,\ldots, y_n)\in \mathbb{P}^n(k)$, $x_i,y_i\in k$, $P\neq Q$, and $v\in M_k$, we define
\begin{equation*}
h_{Q,v}(P)=\log\frac{\max_i \|x_i\|_v \max_i \|y_i\|_v}{\max_{i,j}\|x_iy_j-x_jy_i\|_v}.
\end{equation*}

If $D_1,\ldots, D_q$ are effective Cartier divisors on a projective variety $X$, then we say that $D_1,\ldots, D_q$ are in $m$-subgeneral position if for any subset $I\subset \{1,\ldots, q\}$, $|I|\leq m+1$, we have $\dim \cap_{i\in I}\Supp D_i\leq m-|I|$, where we set $\dim \emptyset =-1$.  In particular, the supports of any $m+1$ divisors in $m$-subgeneral position have empty intersection.  We say that the divisors are in general position if they are in $\dim X$-subgeneral position, i.e., for any subset $I\subset \{1,\ldots, q\}$, $|I|\leq \dim X+1$, we have $\codim \cap_{i\in I}\Supp D_i\geq |I|$.

We now recall three fundamental results in Diophantine approximation on projective space: Roth's theorem, Schmidt's Subspace Theorem, and the Ru-Wong theorem.

To begin, we give a slightly more general version of Roth's theorem from the introduction.
\begin{theorem}[Roth's theorem with multiplicities]
\label{Roth}
Let $S$ be a finite set of places of a number field $k$.  Let $P_1,\ldots, P_q\in \mathbb{P}^1(k)$ be distinct points and let $c_1,\ldots, c_q$ be positive real numbers with $c_1\geq c_2\geq \cdots \geq c_q$.  Let $\epsilon>0$.  Then 
\begin{align*}
\sum_{i=1}^q c_im_{P_i,S}(P)< (c_1+c_2+\epsilon)h(P)+O(1)
\end{align*}
for all $P\in \mathbb{P}^1(k)\setminus \{P_1,\ldots, P_q\}$.
\end{theorem}
\begin{proof}
For all $P\in \mathbb{P}^1(k)\setminus \{P_1,\ldots, P_q\}$,
\begin{align*}
\sum_{i=1}^qc_im_{P_i,S}(P)&\leq (c_1-c_2)m_{P_1,S}(P)+c_2\sum_{i=1}^qm_{P_i,S}(P)+O(1)\\
&\leq (c_1-c_2)h(P)+c_2\sum_{i=1}^qm_{P_i,S}(P)+O(1).
\end{align*}
Let $\epsilon>0$.  By the standard version of Roth's theorem (Theorem \ref{tRoth}), 
\begin{align*}
\sum_{i=1}^qm_{P_i,S}(P)\leq (2+\epsilon)h(P)+O(1)
\end{align*}
for all $P\in\mathbb{P}^1(k)\setminus \{P_1,\ldots, P_q\}$.  So
\begin{align*}
\sum_{i=1}^q c_im_{P_i,S}(P)&\leq (c_1-c_2)h(P)+c_2(2+\epsilon/c_2)h(P)+O(1)\\
&\leq (c_1+c_2+\epsilon)h(P)+O(1)
\end{align*}
for all $P\in \mathbb{P}^1(k)\setminus \{P_1,\ldots, P_q\}$.
\end{proof}

Schmidt's Subspace Theorem is a powerful generalization of Roth's theorem to higher-dimensional projective space.  We state a general version, including improvements due to Schlickewei \cite{Schl}.

\begin{theorem}[Schmidt Subspace Theorem]
\label{subspace}
Let $S$ be a finite set of places of a number field $k$.  For each $v\in S$, let $H_{0,v},\ldots,H_{n,v}\subset \mathbb{P}^n$ be hyperplanes over $k$ in general position.  Let $\epsilon>0$.  Then there exists a finite union of hyperplanes $Z\subset \mathbb{P}^n$ such that  the inequality
\begin{equation*}
\sum_{v\in S}\sum_{i=0}^n h_{H_{i,v},v}(P)< (n+1+\epsilon)h(P)
\end{equation*}
holds for all $P\in \mathbb{P}^n(k)\setminus Z$.
\end{theorem}

If $H_1,\ldots, H_q$ are hyperplanes over $k$ in general position, then the Subspace Theorem easily implies that there exists a finite union of hyperplanes $Z\subset \mathbb{P}^n$ such that  the inequality
\begin{equation*}
\sum_{i=1}^q m_{H_i,S}(P)< (n+1+\epsilon)h(P)
\end{equation*}
holds for all $P\in \mathbb{P}^n(k)\setminus Z$.  If one substitutes a weaker inequality, then the exceptional hyperplanes may be replaced by smaller-dimensional linear subvarieties.  This is given in the Ru-Wong theorem \cite{RW}, which we state more generally for hyperplanes in $m$-subgeneral position.

\begin{theorem}[Ru-Wong]
Let $S$ be a finite set of places of a number field $k$.  Let $H_1,\ldots, H_q\subset \mathbb{P}^n$ be hyperplanes over $k$ in $m$-subgeneral position.  Let $t>2m-n+1$ be a real number.  Then there exists a finite union of linear subvarieties $Z\subset \mathbb{P}^n$ of dimension $\leq 2m+1-t$ such that
\begin{equation*}
\sum_{i=1}^q m_{H_i,S}(P)< th(P)
\end{equation*}
for all $P\in \mathbb{P}^n(k)\setminus (Z\cup H_1\cup\cdots \cup H_q)$.
\end{theorem}

\section{Points of bounded degree and symmetric powers}

For a variety $X$, let $\Sym^d X$ denote the $d$th symmetric power of $X$.  As is well known, $\Sym^d\mathbb{P}^1\cong \mathbb{P}^d$.  In this section we will explore the natural relationship between degree $d$ points on $\mathbb{P}^1$ and rational points on $\Sym^d\mathbb{P}^1\cong \mathbb{P}^d$.

Let $d$ be a positive integer.  Let 
\begin{equation*}
\prod_{i=1}^d b_ix-a_iy=\sum_{i=0}^d p_i(a_1,\ldots, a_d,b_1,\ldots,b_d)x^iy^{d-i},
\end{equation*}
where $p_0,\ldots, p_d$ are polynomials over $\mathbb{Z}$.  We can define a morphism 
\begin{align*}
\sigma:(\mathbb{P}^1)^d&\to \mathbb{P}^d\\
(a_1,b_1)\times \cdots \times (a_d,b_d)&\mapsto (p_0(a_1,\ldots, a_d,b_1,\ldots,b_d),\ldots, p_d(a_1,\ldots, a_d,b_1,\ldots,b_d)).
\end{align*}
The morphism $\sigma$ is a realization of the natural map $(\mathbb{P}^1)^d\to \Sym^d\mathbb{P}^1\cong \mathbb{P}^d$.  

To a point $P=(a,b)\in \mathbb{P}^1(\Qbar)$ we associate the hyperplane $H_P$ in $\mathbb{P}^d$ defined by $\sum_{i=0}^d a^ib^{d-i}x_i=0$.  Since the relevant Vandermonde determinants are nonzero, we find that
\begin{lemma}
If $P_1,\ldots, P_q\in\mathbb{P}^1(\Qbar)$ are distinct points, then the hyperplanes $H_{P_1},\ldots, H_{P_q}$ are in general position.
\end{lemma}

Let $\pi_i:(\mathbb{P}^1)^d\to\mathbb{P}^1$ denote the natural projection map onto the $i$th factor.  

\begin{lemma}
Let $P\in \mathbb{P}^1(\Qbar)$.  Then for any $i$, $\sigma_*\pi_i^*(P)$ is the hyperplane $H_P$.
\end{lemma}

\begin{proof}
By symmetry, it suffices to the prove the lemma for $i=1$.  Let $P=(a,b)$.  Setting $x=a$ and $y=b$, for any $a_2,\ldots, a_d,b_2,\ldots, b_d\in \Qbar$ we have
\begin{equation*}
(bx-ay)\prod_{i=2}^d b_ix-a_iy=\sum_{i=0}^d p_i(a,a_2,\ldots, a_d,b,b_2,\ldots,b_d)a^ib^{d-i}=0.
\end{equation*}
So $\sigma((a,b)\times (a_2,b_2)\times \cdots \times (a_d,b_d))\in H_P$.
Conversely, if
\begin{equation*}
\sum_{i=0}^d c_ia^ib^{d-i}=0,
\end{equation*}
then
\begin{equation*}
\sum_{i=0}^d c_ix^iy^{d-i}=(bx-ay)\prod_{i=2}^d b_ix-a_iy
\end{equation*}
for some $a_2, \ldots, a_d,b_2,\ldots, b_d\in\Qbar$, and hence $\sigma((a,b)\times (a_2,b_2)\times \cdots \times (a_d,b_d))=(c_0,\ldots, c_d)$.  It follows that $\sigma_*\pi_1^*(P)=H_P$.
\end{proof}

Let $k$ be a number field.  For $Q\in\{P\in \mathbb{P}^1(\kbar)\mid [k(P):k]=d\}$, let $Q=Q_1,\ldots, Q_d\in \mathbb{P}^1(\kbar)$ be the $d$ conjugates of $Q$ over $k$ (in some order) and let $\rho(Q)=(Q_1,\ldots,Q_d)\in (\mathbb{P}^1)^d$.  Let $\psi=\sigma\circ \rho:\{P\in \mathbb{P}^1(\kbar)\mid [k(P):k]=d\}\to \mathbb{P}^d(\kbar)$.  Explicitly, if $P=(\alpha,1)$ and $[k(P):k]=d$, then $\psi(P)=(c_0,\ldots, c_d)$ where $\sum_{i=0}^dc_ix^i$ is the minimal polynomial of $\alpha$ over $k$.  The next lemma relates Diophantine approximation on $\mathbb{P}^1$ with respect to $P_1,\ldots, P_q$ and Diophantine approximation on $\mathbb{P}^d$ with respect to $H_{P_1},\ldots, H_{P_q}$.

\begin{lemma}
\label{proxlem}
Let $P_1,\ldots, P_q\in\mathbb{P}^1(k)$.  Then for $Q\in\{P\in \mathbb{P}^1(\kbar)\mid [k(P):k]=d\}$, the point $\psi(Q)$ is $k$-rational and
\begin{align*}
\sum_{i=1}^q m_{H_{P_i},S}(\psi(Q))&=d\sum_{i=1}^q m_{P_i,S}(Q)+O(1),\\
h(\psi(Q))&=d h(Q)+O(1).
\end{align*}
\end{lemma}

\begin{proof}
Let $Q\in\{P\in \mathbb{P}^1(\kbar)\mid [k(P):k]=d\}$ and let $Q=Q_1,\ldots, Q_d\in \mathbb{P}^1(\kbar)$ be the $d$ conjugates of $Q$ over $k$.  It's clear from the definitions (or the remark before Lemma \ref{proxlem}) that $\psi(Q)$ is $k$-rational.  We have, up to $O(1)$,
\begin{align*}
\sum_{i=1}^q m_{H_{P_i},S}(\psi(Q))&=\sum_{i=1}^q m_{\sigma_*\pi_1^*(P_i),S}(\sigma(\rho(Q))=\sum_{i=1}^q m_{\sigma^*\sigma_*\pi_1^*(P_i),S}(\rho(Q))\\
&=\sum_{i=1}^q m_{\sum_{j=1}^d \pi_j^*(P_i),S}(\rho(Q))=\sum_{i=1}^q \sum_{j=1}^d m_{\pi_j^*(P_i),S}(\rho(Q))\\
&=\sum_{i=1}^q \sum_{j=1}^d m_{P_i,S}(\pi_j(\rho(Q)))=\sum_{i=1}^q \sum_{j=1}^d m_{P_i,S}(Q_j)\\
&=d\sum_{i=1}^q m_{P_i,S}(Q).
\end{align*}
A similar calculation shows that $h(\psi(Q))=d h(Q)+O(1)$.
\end{proof}

We end by discussing the relationship between lines in $\Sym^d \mathbb{P}^1$ and morphisms $\phi:\mathbb{P}^1\to\mathbb{P}^1$.

\begin{lemma}
\label{invlem}
Let $P=(a_0,\ldots,a_d),Q=(b_0,\ldots, b_d)\in \mathbb{P}^d$, $P\neq Q$.  Let $L$ be the line through $P$ and $Q$ and let $\phi_{PQ}=\frac{\sum_{i=0}^d a_ix^i}{\sum_{i=0}^d b_ix^i}$.  Then
\begin{align*}
\psi^{-1}(L(k))\subset \phi_{PQ}^{-1}(\mathbb{P}^1(k)).
\end{align*}
\end{lemma}

\begin{proof}
If $d=1$ then the lemma is essentially trivial.  Suppose that $d>1$.  Let $P'\in L(k)$, $P'\neq Q$.  Then 
\begin{align*}
P'=(a_0+tb_0,\ldots, a_d+tb_d)
\end{align*}
for some $t\in k$.  Let $f(x)=\sum_{i=0}^d (a_i+tb_i)x^i$.  If $P'$ is in the image of $\psi,$ then $f$ must be irreducible over $k$ and $\psi^{-1}(P')=\{\alpha_1,\ldots, \alpha_d\}$ is the set of roots of $f$ (identifying $\mathbb{A}^1\subset \mathbb{P}^1$ as usual).  We finish by noting that $\{\alpha_1,\ldots, \alpha_d\}=\phi_{PQ}^{-1}(-t)$.
\end{proof}

\section{Proof of Theorems \ref{tbor}, \ref{tQ}, \ref{Wthm}}

We begin by proving Theorem \ref{tbor}.

\begin{proof}[Proof of Theorem \ref{tbor}]
We first prove part \eqref{p1}.  After an automorphism of $\mathbb{P}^1$, we can assume that $Q_1=0$ and $Q_2=\infty$.  Let $R=\phi^{-1}(\O_{k,S}^*)$.  Since $|S|>1$, the set $R$ is infinite.  From the definitions, for all $P\in R$,
\begin{align*}
m_{Q_1+Q_2,S}(\phi(P))=2h(\phi(P)),
\end{align*}
and by functoriality,
\begin{align*}
m_{\phi^*(Q_1)+\phi^*(Q_2),S}(P)=2dh(P)+O(1).
\end{align*}
For any point $Q\in \mathbb{P}^1(\kbar)$, $m_{Q,S}(P)\leq h(P)+O(1)$.  It follows that for any point $Q\in \phi^{-1}(\{Q_1,Q_2\})$, $m_{Q,S}(P)=h(P)+O(1)$ for all $P\in R$ (i.e., $R$ is a set of $(\phi^*(Q_1)+\phi^*(Q_2),S)$-integral points).  Thus,
\begin{align*}
m_{D,S}(P)\geq(n_1+n_2)h(P)+O(1)
\end{align*}
for all $P\in R$, proving part \eqref{p1}.

We now prove part \eqref{p2}.  We note the symmetry $h_{P,v}(Q)=h_{Q,v}(P)$ for $P,Q\in\mathbb{P}^1(k)$, $P\neq Q$, and $v\in M_k$.  Let $P'\in \phi^{-1}(\mathbb{P}^1(k))$ with $[k(P'):k]=d$.  Let $P_1',\ldots, P_d'$ be the $d$ conjugates of $P'$ over $k$.  Let $i\in \{1,\ldots, q\}$ and let $\phi(P_i)=Q_j$.  Then
\begin{align*}
m_{P_i,S}(P')&=\frac{1}{d}\sum_{j=1}^d m_{P_i,S}(P_j')=\frac{1}{d}\sum_{j=1}^d m_{P_j',S}(P_i)=\frac{1}{d}m_{\phi^*(\phi(P')),S}(P_i)+O(1)\\
&=\frac{1}{d}m_{\phi(P'),S}(\phi(P_i))+O(1)=\frac{1}{d}m_{\phi(P'),S}(Q_j)+O(1)\\
&=\frac{1}{d}m_{Q_j,S}(\phi(P'))+O(1).
\end{align*}
Note also that $h(\phi(P'))=dh(P')+O(1)$.  Let $\epsilon>0$.  Then by Theorem \ref{Roth},
\begin{align*}
m_{D,S}(P')&=\frac{1}{d}\sum_{j=1}^rn_jm_{Q_j,S}(\phi(P'))+O(1)\leq \frac{n_1+n_2+\epsilon}{d}h(\phi(P'))+O(1)\\
&\leq(n_1+n_2+\epsilon)h(P')+O(1).
\end{align*}
\end{proof}

The proof of Theorem \ref{Wthm} proceeds by first transporting the problem to $\Sym^d\mathbb{P}^1\cong \mathbb{P}^d$.  We then use the Ru-Wong theorem to reduce to considering lines in $\mathbb{P}^d$, where Roth's theorem is applicable.

\begin{proof}[Proof of Theorem \ref{Wthm}]
Let $t>2d-1$ be a real number.  If $t>2d$, then the statement in the theorem is an immediate consequence of Wirsing's theorem.  Assume now that $2d-1<t\leq 2d$.  By Wirsing's theorem, inequality \eqref{Wineq} holds for all but finitely many points $P\in \mathbb{P}^1(\kbar)\setminus \Supp D$ satisfying $[k(P):k]< d$.  So we need only consider points  $P\in \mathbb{P}^1(\kbar)$ with $[k(P):k]=d$.  Let
\begin{align*}
R=\{P\in\mathbb{P}^1(\kbar)\mid [k(P):k]=d, m_{D,S}(P)\geq th(P)\}.
\end{align*}
By Lemma \ref{proxlem}, for some constant $C$ we have
\begin{align*}
\sum_{i=1}^q m_{H_{P_i},S}(\psi(P))\geq th(\psi(P))+C
\end{align*}
for all points $P\in R$.  Let $\epsilon>0$ be such that $2d-1+\epsilon<t$.  By the Ru-Wong theorem,
\begin{align*}
\sum_{i=1}^q m_{H_{P_i},S}(P)<(2d-1+\epsilon)h(P)+C
\end{align*}
for all $P\in \mathbb{P}^d(k)\setminus (Z'\cup H_1\cup \cdots \cup H_q)$, where $Z'$ is a finite union of lines and points in $\mathbb{P}^d$ not contained in any of the hyperplanes $H_{P_i}$, $i=1,\ldots, q$.  If $P\in\mathbb{P}^1(\kbar)$ and $[k(P):k]=d$, then $\psi(P)\not\in H_{P_i}$ for all $i$.  Thus, $\psi(R)\subset Z'$ and we need only analyze the set $Z'$.  Let $L$ be a line in the exceptional set $Z'$.  If $L$ is not defined over $k$, then $L(k)$ is finite and may be replaced by a finite number of points in $Z'$.  Assume now that $L$ is defined over $k$.  Let $D=\sum_{i=1}^qH_{P_i}|_L=\sum_{i=1}^s c_iQ_i$, a divisor on $L\cong \mathbb{P}^1$, where $Q_1,\ldots, Q_s\in L(k)$ are distinct points.  Since the hyperplanes $H_{P_i}$ are in general position, $c_i\leq d$ for all $i$.  By Theorem \ref{Roth}, if there are not two distinct indices $j,j'\in \{1,\ldots, s\}$ with $c_j=c_{j'}=d$, then for all $P\in L(k)\setminus \Supp D$,
\begin{align*}
\sum_{i=1}^q m_{H_{P_i},S}(P)=m_{D,S}(P)+O(1)<\left(2d-1+\frac{\epsilon}{2}\right)h(P)+O(1).
\end{align*}
Then again $L$ may be replaced in $Z'$ by a finite number of points.  So assume now that $c_j=c_{j'}=d$ for distinct $j,j'\in \{1,\ldots, s\}$.

Let
\begin{align*}
I_1&=\{i\in \{1,\ldots, q\}\mid Q_j\in H_{P_i}\},\\
I_2&=\{i\in \{1,\ldots, q\}\mid Q_{j'}\in H_{P_i}\}.
\end{align*}
Then by our assumptions, $|I_1|=|I_2|=d$.  Let $P_i=(a_i,b_i)$, $i=1,\ldots, q$.  Let $Q_j=(c_0,\ldots, c_d)$ and $Q_{j'}=(c_0',\ldots, c_d')$.  Let $f_1(x,y)=\sum_{i=0}^d c_ix^iy^{d-i}$ and $f_2(x,y)=\sum_{i=0}^d c_i'x^iy^{d-i}$.  Since $Q_j\in \cap_{i\in I_1}H_{P_i}$,
\begin{align*}
f_1(a_i,b_i)=\sum_{l=0}^d c_la_i^l b_i^{d-l}=0
\end{align*}
for all $i\in I_1$. Similarly, $f_2$ vanishes at $P_i$ for all $i\in I_2$.  Thus, if $\phi=(f_1,f_2)$, then $\phi\in \Phi(D,d,t,k)$.  It follows from Lemma \ref{invlem} that if $\psi(P)\in L(k)$, then $P\in \phi^{-1}(k)$.  Therefore $R\setminus Z(D,d,t,k)$ is a finite set.

Finally, we note that $Z(D,d,t,k)$ admits a simple description.  If $\deg D=q<2d$ then $Z(D,d,t,k)=\emptyset$.  Otherwise, let $I=(I_1,I_2)$, where $I_1$ and $I_2$ are nonempty disjoint subsets of $\{1,\ldots, q\}$ of cardinality $d$.  Then we define $\phi_I=(\prod_{i\in I_1}b_ix-a_iy,\prod_{i\in I_2}b_ix-a_iy)$.  Let $\mathcal{I}$ be the set of all such $I$.  If $2d-1<t\leq 2d$, then
\begin{align*}
Z(D,d,t,k)=\bigcup_{I\in \mathcal{I}}\phi_I^{-1}(\mathbb{P}^1(k)).
\end{align*}
Note that $|\mathcal{I}|=\binom{q}{d,d,q-2d}=\frac{q!}{d!d!(q-2d)!}$ and $\mathcal{I}$ is a finite set.

\end{proof}

Finally, we note that Theorem \ref{tQ} is an immediate consequence of Theorem \ref{Wthm} and the following lemma showing that Question \ref{mq} has a positive answer for trivial reasons when $t\leq d+1$.

\begin{lemma}
\label{rem}
Let $k$ be a number field.  Let $P_1,\ldots, P_q\in \mathbb{P}^1(k)$ be distinct points, let $D=\sum_{i=1}^qP_i$, and let $t$ be a positive real number.  
\begin{enumerate}
\item \label{la} Let $S$ be a finite set of places of $k$.  If $\deg D<t$, then
\begin{equation*}
m_{D,S}(P)< th(P)
\end{equation*}
for all but finitely many points $P\in \mathbb{P}^1(\kbar)$.
\item \label{lb} If $t\leq d+1$ and $t\leq \deg D$, then 
\begin{align*}
Z(D,d,t,k)=\{P\in\mathbb{P}^1(\kbar)\mid [k(P):k]\leq d\}.
\end{align*}
\end{enumerate}
\end{lemma}

\begin{proof}
Part \eqref{la} follows from the trivial observation that if $\deg D<t$, then 
\begin{align*}
m_{D,S}(P)\leq h_D(P)+O(1)=(\deg D)h(P)+O(1)<th(P) 
\end{align*}
for all but finitely many $P\in\mathbb{P}^1(\kbar)$.

To prove \eqref{lb}, suppose now that $t\leq d+1$ and $t\leq \deg D$.  Without loss of generality we can assume that $t$ is a positive integer.  One of the set inclusions in the statement is trivial.  For the other, let $P\in \mathbb{P}^1(\kbar)$ with $[k(P):k]\leq d$.  Let $P_i=(\alpha_i,1)$ and $P=(\alpha,1)$, where $\alpha_i\in k$, $i=1,\ldots, t$, and $\alpha\in \{x\in\kbar\mid [k(x):k]\leq d\}$ (after an automorphism, we can assume that none of the points are the point at infinity).  If $\alpha\in \{\alpha_1,\ldots, \alpha_t\}$, then it is easy that $P\in Z(D,d,t,k)$.  Otherwise, let $\phi_0=\frac{\prod_{i=1}^{t-1}x-\alpha_i}{x-\alpha_t}$.
Since $[k(\alpha):k]\leq d$ and $\phi_0(\alpha)\in k(\alpha)$, we can write $\phi_0(\alpha)=\sum_{i=0}^{d-1}c_i\alpha^i$ with $c_i\in k$, $i=0,\ldots, d-1$.  If $[k(\alpha):k]<d$, then we have some freedom in choosing the $c_i$.  In any case, we can ensure that none of $\alpha_1,\ldots, \alpha_t$ are roots of $\sum_{i=0}^{d-1}c_ix^i$.  Now let $\phi=\phi_0/\sum_{i=0}^{d-1}c_ix^i$.  Then $\phi(\alpha)=1$, $\deg \phi\leq d$, $\phi\in \End_k(\mathbb{P}^1)$, and $|\phi^{-1}(\{0,\infty\})\cap \Supp D|\geq t$.  So $\phi\in \Phi(D,d,t,k)$ and $P\in Z(D,d,t,k)$.

\end{proof}

\section{Exceptional subspaces in $\mathbb{P}^3$}

In order to prove Theorem \ref{td3} we need to study the exceptional hyperplanes that appear in the Schmidt Subspace Theorem for hyperplanes $H_1,\ldots, H_q$ in $\mathbb{P}^3$ in general position.  If $H_1,\ldots, H_q$ are hyperplanes in $\mathbb{P}^3$ in general position and $H$ is a hyperplane in $\mathbb{P}^3$ distinct from $H_1,\ldots, H_q$, then $H_1\cap H,\ldots, H_q\cap H$ are lines in $H\cong \mathbb{P}^2$ in $3$-subgeneral position.  Thus, we are reduced to studying Diophantine approximation in the plane with respect to lines in $3$-subgeneral position.

Let $L_1,\ldots, L_q$ be lines in $\mathbb{P}^2$ in $3$-subgeneral position.  We say that $L_1,\ldots, L_q$ is of:
\begin{enumerate}
\item Type I if $q>4$ and 
\begin{enumerate}
\item $L_i=L_j$ for some $i\neq j$.
\item There is a point in $\mathbb{P}^2$ that is contained in three distinct lines in $\{L_1,\ldots, L_q\}$.
\end{enumerate}
\item Type II if $q>4$ and 
\begin{enumerate}
\item The lines $L_1,\ldots, L_q$ are distinct.
\item There are at least three noncollinear points in $\mathbb{P}^2$ that are each contained in three distinct lines in $\{L_1,\ldots, L_q\}$.
\end{enumerate}
\item Type III otherwise.
\end{enumerate}

Define
\begin{align*}
c(L_1,\ldots, L_q)=
\begin{cases}
5 &\text{if $L_1,\ldots, L_q$ is of Type I},\\
\frac{9}{2} &\text{if $L_1,\ldots, L_q$ is of Type II},\\
4 &\text{if $L_1,\ldots, L_q$ is of Type III}.
\end{cases}
\end{align*}

\begin{theorem}
\label{3subthm}
Let $k$ be a number field and let $S$ be a finite set of places of $k$.  Let $L_1,\ldots, L_q\subset \mathbb{P}^2$ be lines over $k$ in $3$-subgeneral position.  Let $c=c(L_1,\ldots, L_q)$ and let $\epsilon>0$.  Then there exists a finite union of lines $Z$ in $\mathbb{P}^2$ such that
\begin{equation*}
\sum_{i=1}^q m_{L_i,S}(P)\leq (c+\epsilon)h(P)
\end{equation*}
for all points $P\in \mathbb{P}^2(k)\setminus Z$.
\end{theorem}

\begin{proof}
By the Ru-Wong theorem, there exists a finite union of lines $Z$ in $\mathbb{P}^2$ such that
\begin{equation*}
\sum_{i=1}^q m_{L_i,S}(P)\leq (5+\epsilon)h(P)
\end{equation*}
for all points $P\in \mathbb{P}^2(k)\setminus Z$.  So if $L_1,\ldots, L_q$ is of Type I we are done.  Suppose now that $L_1,\ldots, L_q$ is of Type II.  Since the lines $L_1,\ldots, L_q$ are in $3$-subgeneral position, any point can be $v$-adically close to at most three of the lines $L_1,\ldots, L_q$.  It follows that
\begin{equation*}
\sum_{i=1}^q m_{L_i,S}(P)=\sum_{v\in S}\sum_{i=1}^q h_{L_i,v}(P)\leq \sum_{v\in S}\sum_{i=1}^3 h_{L_{i,v},v}(P)+O(1),
\end{equation*}
where for each $v\in S$, $L_{1,v},L_{2,v},L_{3,v}$ are some choice of distinct lines in $\{L_1,\ldots, L_q\}$.  Then by the Schmidt Subspace Theorem, for all $\epsilon>0$, there exists a finite union of lines $Z$ in $\mathbb{P}^2$ such that
\begin{align*}
\sum_{v\in S} h_{L_{1,v},v}(P)+h_{L_{2,v},v}(P)\leq (3+\epsilon)h(P),\\
\sum_{v\in S} h_{L_{1,v},v}(P)+h_{L_{3,v},v}(P)\leq (3+\epsilon)h(P),\\
\sum_{v\in S} h_{L_{2,v},v}(P)+h_{L_{3,v},v}(P)\leq (3+\epsilon)h(P),
\end{align*}
for all $P\in\mathbb{P}^2(k)\setminus Z$.  Adding the three equations and dividing by $2$ yields that for all $\epsilon>0$, there exists a finite union of lines $Z$ in $\mathbb{P}^2$ such that
\begin{equation*}
\sum_{i=1}^q m_{L_i,S}(P)\leq \left(\frac{9}{2}+\epsilon\right)h(P)
\end{equation*}
for all $P\in \mathbb{P}^2(k)\setminus Z$, as desired.

Finally, suppose that $L_1,\ldots, L_q$ is of Type III.  If $q\leq 4$, then it is trivial that
\begin{equation*}
\sum_{i=1}^q m_{L_i,S}(P)\leq (4+\epsilon)h(P).
\end{equation*}
Suppose now that $q>4$.  Suppose that some line appears twice in $L_1,\ldots, L_q$.  Then there must be exactly one such line (from $3$-subgeneral position) and since $L_1,\ldots, L_q$ is not of Type I, no three distinct lines in $\{L_1,\ldots, L_q\}$ meet at a point.  After reindexing, we may assume that $L_{q-1}=L_q$.  Then it follows that the lines $L_1,\ldots, L_{q-1}$ are in general position.  Let $\epsilon>0$.  Then by the Schmidt Subspace Theorem, there exists a finite union of lines $Z$ in $\mathbb{P}^2$ such that
\begin{align*}
\sum_{i=1}^q m_{L_i,S}(P)&=\sum_{i=1}^{q-1} m_{L_i,S}(P)+m_{L_q,S}(P)\leq \sum_{i=1}^{q-1} m_{L_i,S}(P)+h(P)\\
&\leq (4+\epsilon)h(P)
\end{align*}
for all $P\in \mathbb{P}^2(k)\setminus Z$.

We now assume that $L_1,\ldots, L_q$ are distinct lines.  Let $P_1,\ldots, P_n$ be the points in $\mathbb{P}^2$ that are contained in three distinct lines in $\{L_1,\ldots, L_q\}$.  Then since $L_1,\ldots, L_q$ is not of Type II, $P_1,\ldots, P_n$ all lie on a line $L$.  Let $v\in S$ and $P\in \mathbb{P}^2(k)\setminus \cup_{i=1}^qL_i$.  For simplicity, rearrange the indices so that
\begin{align*}
h_{L_1,v}(P)\geq h_{L_2,v}(P)\geq \cdots \geq h_{L_q,v}(P).
\end{align*}
If $L_1\cap L_2\neq \{P_i\}$, $i=1,\ldots, n$, then
\begin{align*}
\sum_{i=1}^q h_{L_i,v}(P)\leq h_{L_1,v}(P)+h_{L_2,v}(P)+O(1).
\end{align*}
If $L_1\cap L_2\cap L_j=\{P_i\}$ for some $i\in \{1,\ldots, n\}$ and $j\in \{3,\ldots, q\}$, then from the theory of heights associated to closed subschemes \cite{Sil}, we have
\begin{align*}
\min \{h_{L_1,v}(P),h_{L_2,v}(P), h_{L_{j'},v}(P)\}=
\begin{cases}
h_{P_i,v}(P)+O(1) \quad & \text{if $j'=j$},\\
O(1) & \text{if }j'\not\in\{1,2,j\},
\end{cases}
\end{align*}
and if $P\not\in L$,
\begin{align*}
h_{P_i,v}(P)\leq h_{L,v}(P)+O(1).
\end{align*}
Then if $P\not\in L$,
\begin{align*}
\sum_{i=1}^q h_{L_i,v}(P)&\leq h_{L_1,v}(P)+h_{L_2,v}(P)+h_{P_i,v}(P)+O(1)\\
&\leq h_{L_1,v}(P)+h_{L_2,v}(P)+h_{L,v}(P)+O(1).
\end{align*}
It follows that if $P\not\in L$,
\begin{align*}
\sum_{i=1}^q m_{L_i,S}(P)=\sum_{v\in S}\sum_{i=1}^q h_{L_i,v}(P)\leq \sum_{v\in S}\sum_{i=1}^2 h_{L_{i,v},v}(P)+\sum_{v\in S}h_{L,v}(P)+O(1)
\end{align*}
for some lines $L_{i,v}$, $v\in S$.  Then by the Schmidt Subspace Theorem and the trivial estimate $\sum_{v\in S}h_{L,v}(P)\leq h(P)+O(1)$, we find that there exists a finite union of lines $Z$ in $\mathbb{P}^2$ such that
\begin{align*}
\sum_{v\in S}\sum_{i=1}^q h_{L_i,v}(P)\leq (4+\epsilon)h(P)
\end{align*}
for all $P\in \mathbb{P}^2(k)\setminus Z$.
\end{proof}

We now show that the previous theorem is essentially sharp.
\begin{theorem}
\label{sharp}
Let $k$ be a number field and let $S$ be a finite set of places of $k$ containing the archimedean places.  Let $L_1,\ldots, L_q\subset \mathbb{P}^2$, $q>3$, be lines over $k$ in $3$-subgeneral position, but not in general position.  Let $c=c(L_1,\ldots, L_q)$.  Suppose that 
\begin{align*}
\begin{cases}
|S|>1 &\text{ if $L_1,\ldots, L_q$ is of Type I or III},\\
|S|>2 &\text{ if $L_1,\ldots, L_q$ is of Type II}.
\end{cases}
\end{align*}
Then there exists a Zariski dense set of points $R\subset\mathbb{P}^2(k)$ such that
\begin{equation*}
\sum_{i=1}^q m_{L_i,S}(P)\geq (c-\epsilon)h(P)
\end{equation*}
for all $P\in R$.
\end{theorem}

\begin{proof}
Suppose first that $L_1,\ldots, L_q$ is of Type I.  Then after reindexing, we can assume that $L_1\cap L_2\cap L_3=\{Q\}$ is nonempty and $L_4=L_5$.  Let $L$ be a $k$-rational line through $Q$ distinct from $L_1,\ldots, L_q$.  Then $\cup_{i=1}^5L\cap L_i=\{Q,Q'\}$ consists of two points.  Since $|S|>1$, there exists an infinite set $R$ of $k$-rational $(Q+Q',S)$-integral points on $L$, i.e.,
\begin{align*}
m_{Q+Q',S}(P)=2h(P)+O(1)
\end{align*}
for all $P\in R$.  Then for all $P\in R$,
\begin{align*}
\sum_{i=1}^q m_{L_i,S}(P)\geq \sum_{i=1}^5 m_{L_i,S}(P)+O(1)=3m_{Q,S}(P)+2m_{Q',S}(P)+O(1)=5h(P)+O(1).
\end{align*}
Thus, there are infinitely many points $P\in L(k)$ satisfying
\begin{equation*}
\sum_{i=1}^q m_{L_i,S}(P)\geq (5-\epsilon)h(P).
\end{equation*}
Since the union of $k$-rational lines $L$ through $Q$ is Zariski dense in $\mathbb{P}^2$, this proves the result in the Type I case.

Suppose now that $L_1,\ldots, L_q$ is of Type III.  Since $L_1,\ldots, L_q$ are not in general position, after reindexing we can assume that $L_1\cap L_2\cap L_3=\{Q\}$ is nonempty.  Let $L$ be a $k$-rational line through $Q$ distinct from $L_1,\ldots, L_q$ and let $\{Q'\}=L\cap L_4$.  Then by the same argument as above, taking $R\subset L(k)$ to be an infinite set of $(Q+Q',S)$-integral points on $L$, for all $P\in R$ we have
\begin{align*}
\sum_{i=1}^q m_{L_i,S}(P)\geq \sum_{i=1}^4 m_{L_i,S}(P)+O(1)=3m_{Q,S}(P)+m_{Q',S}(P)+O(1)=4h(P)+O(1).
\end{align*}
Thus, there are infinitely many points $P\in L(k)$ satisfying
\begin{equation*}
\sum_{i=1}^q m_{L_i,S}(P)\geq (4-\epsilon)h(P).
\end{equation*}
Since the union of such lines $L$ is Zariski dense in $\mathbb{P}^2$, this proves the result in the Type III case.

Finally, suppose that $L_1,\ldots, L_q$ is of Type II. Let $Q_1, Q_2,$ and $Q_3$ be three noncollinear points in $\mathbb{P}^2(k)$ that are each contained in three distinct lines in $\{L_1,\ldots, L_q\}$.  After an automorphism of $\mathbb{P}^2$ we may assume that $Q_1=(1,0,0)$, $Q_2=(0,1,0)$, and $Q_3=(0,0,1)$.  Let $S=\{v_1,\ldots, v_n\}$, where by assumption $n\geq 3$.  From the proof of the Dirichlet unit theorem, for each positive integer $m$, there exist units $u_{1,m},u_{2,m}\in \O_{k,S}^*$ such that
\begin{align*}
\log \|u_{1,m}\|_{v_1}=m+O(1), \log \|u_{1,m}\|_{v_2}=O(1), \log \|u_{1,m}\|_{v_i}=-\frac{m}{n-2}+O(1), i=3,\ldots, n,\\
\log \|u_{2,m}\|_{v_1}=O(1), \log \|u_{2,m}\|_{v_2}=m+O(1), \log \|u_{2,m}\|_{v_i}=-\frac{m}{n-2}+O(1), i=3,\ldots, n.
\end{align*}
Let $P_m=(u_{1,m},u_{2,m},1)\in\mathbb{P}^2(k)$.  Let $L_x,L_y$, and $L_z$ be the three lines in $\mathbb{P}^2$ defined by $x=0$, $y=0$, and $z=0$, respectively.  Then $h(P_m)=2m+O(1)$ and
\begin{align*}
h_{L_x,v_1}(P_m)&=O(1),h_{L_x,v_2}(P_m)=m+O(1), h_{L_x,v_i}(P_m)=\frac{m}{n-2}+O(1), i=3,\ldots, n,\\
h_{L_y,v_1}(P_m)&=m+O(1),h_{L_y,v_2}(P_m)=O(1), h_{L_y,v_i}(P_m)=\frac{m}{n-2}+O(1), i=3,\ldots, n,\\
h_{L_z,v_1}(P_m)&=m+O(1),h_{L_z,v_2}(P_m)=m+O(1), h_{L_z,v_i}(P_m)=O(1), i=3,\ldots, n.
\end{align*}
For $v\in S$, we have (see \cite{Sil})
\begin{align*}
h_{Q_1,v}&=\min \{h_{L_y,v},h_{L_z,v}\}+O(1),\\
h_{Q_2,v}&=\min \{h_{L_x,v},h_{L_z,v}\}+O(1),\\
h_{Q_3,v}&=\min \{h_{L_x,v},h_{L_y,v}\}+O(1),
\end{align*}
where the functions are defined.
It follows that
\begin{align*}
h_{Q_1,v_1}(P_m)&=m+O(1), h_{Q_1,v_2}(P_m)=O(1), h_{Q_1,v_i}(P_m)=O(1), i=3,\ldots, n,\\
h_{Q_2,v_1}(P_m)&=O(1), h_{Q_2,v_2}(P_m)=m+O(1), h_{Q_2,v_i}(P_m)=O(1), i=3,\ldots, n,\\
h_{Q_3,v_1}(P_m)&=O(1), h_{Q_3,v_2}(P_m)=O(1), h_{Q_3,v_i}(P_m)=\frac{m}{n-2}+O(1), i=3,\ldots, n.
\end{align*}
Then for all $m$ such that $P_m\not\in L_1\cup \cdots \cup L_q$,
\begin{align*}
\sum_{i=1}^q m_{L_i,S}(P_m)\geq 3 \sum_{v\in S}\sum_{i=1}^3h_{Q_i,v}(P_m)=9m+O(1)=\frac{9}{2}h(P_m)+O(1).
\end{align*}

To complete the proof, it remains to show that the set $R=\{P_m\mid m\in \mathbb{N}\}$ is Zariski dense in $\mathbb{P}^2$.  Suppose that there exists a homogeneous polynomial $p\in k[x,y,z]$ that vanishes on $R$.  Looking at the valuations of $u_{1,m}^iu_{2,m}^j$ with respect to $v_1$ and $v_2$, this is plainly impossible.  Thus, we arrive at a contradiction and the set $R$ is Zariski dense in $\mathbb{P}^2$.
\end{proof}

\section{Proof of Theorem \ref{td3}}

Using the results of the previous section we now prove Theorem \ref{td3}.

\begin{proof}[Proof of Theorem \ref{td3}]
We first prove part \eqref{Wa}.  If $t>5$, then part \eqref{Wa} follows immediately from Theorem \ref{Wthm}.  Suppose now that $\frac{9}{2}<t\leq 5$.  By Wirsing's theorem, the set of points $P\in\mathbb{P}^1(\kbar)\setminus \Supp D$  satisfying $[k(P):k]\leq 2$ and
\begin{equation*}
m_{D,S}(P)\geq th(P)
\end{equation*}  
is finite, and so we may ignore such points.  Let $R$ be the set
\begin{equation*}
R=\{P\in\mathbb{P}^1(\kbar)\mid [k(P):k]=3, m_{D,S}(P)\geq th(P)\}.
\end{equation*}

Then by Lemma \ref{proxlem},
\begin{equation*}
\sum_{i=1}^{q}m_{H_{P_i,S}}(\psi(P))\geq th(\psi(P))+O(1)
\end{equation*}
for all $P\in R$.  Let $R'=\psi(R)$.  Since $t>4$, by the Schmidt Subspace Theorem, $R'$ lies in a finite union of hyperplanes of $\mathbb{P}^3$.  Let $H$ be one of the hyperplanes.

Suppose first that $H_{P_1}|_H,\ldots, H_{P_q}|_H$ is not of Type I.  Then by Theorem \ref{3subthm}, $R'\cap H$ lies in a finite union of lines (with no line contained in any of the hyperplanes $H_{P_1},\ldots, H_{P_q}$).  Let $L$ be one of these lines and let $
\sum_{i=1}^q H_{P_i}|_L=\sum_{i=1}^rc_i Q_i$, where $Q_1,\ldots, Q_r\in L(k)$ are distinct points and $c_1\geq c_2\geq \cdots \geq c_r$.  Then for all $P\in R'\cap L$,
\begin{align*}
\sum_{i=1}^r c_im_{Q_i,S}(P)\geq th(P)+O(1).
\end{align*}
If $R'\cap L$ is infinite, then by Theorem \ref{Roth}, we must have $c_1+c_2\geq t>\frac{9}{2}$.  Since $c_1$ and $c_2$ are integers and $c_1,c_2\leq 3$, we must have that $c_1=3$ and $c_2\geq 2$.  After reindexing, we can assume that $H_{P_1}\cap H_{P_2}\cap H_{P_3}\cap L=\{Q_1\}$ and $H_{P_4}\cap H_{P_5}\cap L =\{Q_2\}$.  By Lemma \ref{invlem}, $\psi^{-1}(L(k))\subset \phi_{Q_1Q_2}^{-1}(\mathbb{P}^1(k))$.  From the definitions, $\phi_{Q_1Q_2}^{-1}(0)=\{P_1,P_2,P_3\}$ and $\phi_{Q_1Q_2}^{-1}(\infty)\supset\{P_4,P_5\}$.  Thus, since $t\leq 5$, $\phi_{Q_1Q_2}\in \Phi(D,3,t,k)$ and $\psi^{-1}(L(k))\subset Z(D,3,t,k)$.  It follows that all but finitely many points of $\psi^{-1}(R'\cap H)$ are contained in $Z(D,3,t,k)$.

Suppose now that $H_{P_1}|_H,\ldots, H_{P_q}|_H$ is of Type I.  After reindexing, we can assume that $H_{P_1}\cap H_{P_2}\cap H_{P_3}\cap H=\{Q\}$, for some point $Q\in H(k)$, and $H_{P_4}\cap H=H_{P_5}\cap H$.  Let $P\in H(k)\setminus (H_{P_1}\cup H_{P_2}\cup H_{P_3})$, and let $L$ be the line through $P$ and $Q$.  Let $L\cap H_4\cap H=L\cap H_5\cap H=\{Q'\}$.  By Lemma \ref{invlem}, $\psi^{-1}(L(k))\subset \phi_{QQ'}^{-1}(\mathbb{P}^1(k))$.  From the definitions, $\phi_{QQ'}^{-1}(0)=\{P_1,P_2,P_3\}$ and $\phi_{QQ'}^{-1}(\infty)\supset\{P_4,P_5\}$.  Thus, since $t\leq 5$, $\phi_{QQ'}\in \Phi(D,3,t,k)$ and $\psi^{-1}(L(k))\subset Z(D,3,t,k)$.  Since $P\in H(k)\setminus (H_{P_1}\cup H_{P_2}\cup H_{P_3})$ was arbitrary, in particular $\psi^{-1}(R'\cap H)\subset Z(D,3,t,k)$.  Combining this fact with the previous case above, we have shown that $R\setminus Z(D,3,t,k)$ is a finite set, proving part \eqref{Wa}.

Suppose now that $4<t<\frac{9}{2}$, $|S|>2$, and $q=6$.  Let 
\begin{align*}
\{Q_1\}&=H_{P_1}\cap H_{P_2}\cap H_{P_3},\\
\{Q_2\}&=H_{P_1}\cap H_{P_4}\cap H_{P_5},\\
\{Q_3\}&=H_{P_2}\cap H_{P_4}\cap H_{P_6}.
\end{align*}
The line through $Q_1$ and $Q_2$ lies in $H_{P_1}$.   Since the hyperplanes $H_{P_i}$ are in general position, $Q_3\not\in H_{P_1}$ and it follows that $Q_1,Q_2$, and $Q_3$ are not collinear.  Let $H\subset\mathbb{P}^3$ be the unique hyperplane through $Q_1$, $Q_2$, and $Q_3$.  Since the hyperplanes $H_{P_i}$ are in general position, it follows easily that all of the lines $H_{P_i}|_H$ are distinct (otherwise there would be four hyperplanes $H_{P_i}$ containing some point $Q_j$).  Then $H_{P_1}|_H,\ldots, H_{P_6}|_H$ is of Type II.  Let $0<\epsilon<\frac{1}{4}$ be such that $t<\frac{9}{2}-\epsilon$.  By Theorem~\ref{sharp}, there exists a set of points $R'\subset H(k)$ that is Zariski dense in $H$ and such that
\begin{equation*}
\sum_{i=1}^{6}m_{H_{P_i,S}}(P)> \left(\frac{9}{2}-\epsilon\right)h(P)
\end{equation*}
for all $P\in R'$.  Let $P\in R'$ and let $\sigma((Q_1',Q_2',Q_3'))=P$.  Then by the same calculation as in the proof of Lemma \ref{proxlem}, we have
\begin{equation*}
\sum_{i=1}^6\sum_{j=1}^3 m_{P_i,S}(Q_j')> \left(\frac{9}{2}-\epsilon\right)\sum_{j=1}^3 h(Q_j')+O(1).
\end{equation*}

If $[k(Q_j'):k]\leq 2$ for some $j$ (and hence all $j$), then by Wirsing's theorem, $\sum_{i=1}^6 m_{P_i,S}(Q_j')<(4+\epsilon) h(Q_j')+O(1)$.  It follows that for all but finitely many points $P\in R'$, $P\in \im \psi$.  Let $R=\psi^{-1}(R')$.  By Lemma \ref{proxlem},
\begin{equation*}
\sum_{i=1}^6 m_{P_i,S}(P)> \left(\frac{9}{2}-\epsilon\right) h(P)+O(1)>th(P)
\end{equation*}
for all but finitely many $P\in R$.  From the definitions and the proof of Lemma \ref{invlem}, every point in $\psi(R\cap Z(D,3,t,k))$ lies on a line $L$ through points $P$ and $Q$ in $\mathbb{P}^3$, where $P$ lies in the intersection of three distinct hyperplanes $H_{P_{i_1}}, H_{P_{i_2}}, H_{P_{i_3}}$, and $Q$ lies in the intersection of two other distinct hyperplanes $H_{P_{i_4}}$ and $H_{P_{i_5}}$.  The set of such lines $L$ does not intersect $H$ in a Zariski dense set in $H$.  It follows that $R\setminus Z(D,3,t,k)$ is infinite.
\end{proof}

\bibliography{Bounded}
\end{document}